\documentclass{amsproc}
\usepackage{amsfonts}

\theoremstyle{plain}

\newtheorem{corollary}{Corollary}

\newtheorem{definition}{Definition}

\newtheorem{proposition}{Proposition}
\newtheorem{remark}{Remark}

\newtheorem{theorem}{Theorem}
\numberwithin{equation}{section}

\begin{document}
\title{ Generalizations of Certain Classes of Single-Valued Mappings to
Multivalued Cases}
\author{Emirhan HACIOĞLU}
\address{Department of Mathematics, Trakya University, 22030 Edirne, T\~{A}%
$\frac14$%
rkiye }
\email{emirhanhacioglu@trakya.edu.tr}
\thanks{This paper is in final form and no version of it will be submitted
for publication elsewhere.}

\begin{abstract}
In this study, multivalued generalizations of certain classes of
single-valued transformations defined on metric spaces are obtained.
Building upon recently introduced concepts such as mappings contracting
perimeters of triangles, new structural properties and fixed-point results
for the multivalued counterparts are explored. 
\end{abstract}

\maketitle

\section{Introduction}

Fixed point theory for contractive-type mappings remains a vibrant area of
research with broad applications in analysis, optimization, and applied
sciences. The classical Banach contraction\cite{banach} principle laid the
foundation, motivating multiple generalizations such as Kannan-type and
Chatterjea-type contractions \cite{kannan1968,chatterjea1972} Petrov \cite%
{petrov} investigated mappings contracting total pairwise distances of $n$
points, establishing periodic point results in complete metric spaces.
Popescu \cite{popescu} introduced three-point generalized orbital triangular
contractions, proving convergence under minimal continuity assumptions.
However, certain proof details and multivalued cases were left incomplete.

This paper merges and enhances these approaches. Our main contributions
include:

\begin{itemize}
\item Proposing a unified framework that encompasses perimeter-contracting,
total pairwise distance, and generalized orbital triangular contractions,
covering multivalued mappings.

\item Defining and give some existence theoresm for multivalued orbital
triangular contraction theorems .
\end{itemize}

\section{Preliminaries}

\label{sec:prelim} Let $(X,d)$ be a metric space. For nonempty closed
bounded sets $A,B \subset X$, the $\delta$-distance is defined by 
\begin{equation*}
\delta (A,B) = \sup_{a \in A,\, b \in B} d(a,b).
\end{equation*}

We denote by $CB(X)$ the family of all nonempty closed bounded subsets of $X$%
. \bigskip

\begin{definition}
\cite{petrov2} Let $(X,d)$ be a metric space with $|X|>3$. A mapping $%
T:X\rightarrow X$ is said to **contract the perimeters of triangles** if
there exists $\alpha \in \lbrack 0,1)$ such that for all distinct $x,y,z\in
X $, 
\begin{equation}
d(Tx,Ty)+d(Ty,Tz)+d(Tx,Tz)\leq \alpha \left( d(x,y)+d(y,z)+d(x,z)\right) .
\end{equation}
\end{definition}

\begin{remark}
\cite{petrov2} The requirement that $x,y,z$ be pairwise distinct is
essential. Without it, the definition becomes equivalent to contraction
mappings.
\end{remark}

\begin{proposition}
\cite{petrov2} Mappings that contract the perimeters of triangles are
continuous.
\end{proposition}

\begin{theorem}
\cite{petrov2} Let $(X,d)$ be a complete metric space with $|X|>3$, and let $%
T:X\rightarrow X$ satisfy:

\begin{itemize}
\item[(i)] $T(T(x)) \neq x$ whenever $Tx \neq x$,

\item[(ii)] $T$ contracts the perimeters of triangles.
\end{itemize}

Then $T$ has a fixed point. Moreover, the number of fixed points is at most
two.
\end{theorem}

\begin{remark}
\cite{petrov2} If $x^{\ast }$ is the limit of an iteration sequence and was
not fixed earlier, then $x^{\ast }$ is the unique fixed point.
\end{remark}

\begin{definition}
A mapping $T : X \to X$ on a metric space $(X,d)$ is called a \emph{%
contraction mapping} if there exists a constant $\alpha \in [0,1)$ such that
for all $x, y \in X$, 
\begin{equation}
d(Tx, Ty) \leq \alpha\, d(x, y).
\end{equation}
\end{definition}

\section{Multivalued Generalization of Mappings Contracting Total Pairwise
Distances}

\label{sec:prelim}

Let $(X,d)$ be a metric space with $|X|\geq 2$ and $n\geq 2$. Define the
total pairwise distance: 
\begin{equation*}
S(x_{1},\dotsc ,x_{n})=\sum_{1\leq i<j\leq n}d(x_{i},x_{j}).
\end{equation*}

\begin{definition}
\cite{petrov} A mapping $T:X\rightarrow X$ multivalued contracting the total
pairwise distances between $n$ points if there exists $\alpha \in [0,1)$
such that for all pairwise distinct $x_{1},\dotsc ,x_{n}\in X$, 
\begin{equation*}
S(Tx_{1},\dotsc ,Tx_{n})\leq \alpha S(x_{1},\dotsc ,x_{n}).
\end{equation*}
\end{definition}

Let $X_{i}\subseteq CB(X)$ for all $i=1,2,\ldots,n$. We define the total
pairwise distance on $CB(X)$ as follows: 
\begin{equation*}
S(X_{1},\dotsc ,X_{n})=\sum_{1\leq i<j\leq n}\delta (X_{i},X_{j}).
\end{equation*}

\begin{definition}
A mapping $T:X\rightarrow CB(X)$ is multivalued contracting the total
pairwise distances between $n$ points if there exists $\alpha \in [0,1)$
such that 
\begin{equation*}
S(Tx_{1},\dotsc ,Tx_{n})\leq \alpha S(x_{1},\dotsc ,x_{n})
\end{equation*}
for all pairwise distinct $x_{1},\dotsc ,x_{n}\in X$.
\end{definition}

\begin{proposition}
If $T$ satisfies the above inequality for any $n$ points with $%
|\{x_{1},\dotsc ,x_{n}\}|=k$ where $2\leq k\leq n-1$, then $T$ contracts the
total pairwise distances between $k$ points.
\end{proposition}

\begin{proof}
Let $\{x_{1},x_{2},\ldots ,x_{k}\}$ be a set of pairwise distinct points
from $X$. Consider $k$ sets $A_{i}=\{x_{1},x_{2},\ldots
,x_{k},x_{k+1},\ldots ,x_{n}\}$ consisting of $n$ points such that $%
x_{k+1}=\cdots =x_{n}=x_{i}$, for $i=1,\ldots ,k$, i.e., the element $x_{i}$
occurs $n-k+1$ times in the set $A_{i}$. Then, by the supposition, the
inequalities 
\begin{equation*}
S(Tx_{1},Tx_{2},\ldots ,Tx_{k},Tx_{i},\ldots ,Tx_{i})\leq \alpha
S(x_{1},x_{2},\ldots ,x_{k},x_{i},\ldots ,x_{i})
\end{equation*}
hold for all $i=1,\ldots ,k$. Summing both sides over $i$, we get: 
\begin{equation*}
\sum_{i=1}^{k}S(Tx_{1},\ldots ,Tx_{k},Tx_{i},\ldots ,Tx_{i})\leq \alpha
\sum_{i=1}^{k}S(x_{1},\ldots ,x_{k},x_{i},\ldots ,x_{i}).
\end{equation*}
\end{proof}

Note that 
\begin{equation*}
S(Tx_{i},\ldots ,Tx_{i})=S(x_{i},\ldots ,x_{i})=0,
\end{equation*}
and 
\begin{equation*}
\sum_{i=1}^{k}\sum_{j=1}^{k}\delta (Tx_{j},Tx_{i})=2S(Tx_{1},\ldots
,Tx_{k}),\quad \sum_{i=1}^{k}\sum_{j=1}^{k}d(x_{j},x_{i})=2S(x_{1},\ldots
,x_{k}).
\end{equation*}

So, we obtain: 
\begin{equation*}
(k+2(n-k))S(Tx_{1},\ldots ,Tx_{k})\leq \alpha (k+2(n-k))S(x_{1},\ldots
,x_{k}).
\end{equation*}

Dividing both sides by $k+2(n-k)$, we conclude: 
\begin{equation*}
S(Tx_{1},\ldots ,Tx_{k})\leq \alpha S(x_{1},\ldots ,x_{k}).
\end{equation*}

Thus, $T$ is a mapping contracting the total pairwise distances between $k$
points.

\begin{proposition}
If $T$ is a multivalued mapping contracting the total pairwise distances
between $m$ points, $m\geq 2$, then $T$ also contracts the total pairwise
distances between $n$ points for all $n>m$.
\end{proposition}

\begin{proof}
Similar steps are given in \cite{petrov}.
\end{proof}

\begin{theorem}
Let $(X,d)$ be a complete metric space with $|X|\geq n$, and $T:X\rightarrow
CB(X)$ a multivalued mapping contracting the total pairwise distances
between $n$ points. Then $T$ has a periodic point of prime period $k\in
\{1,\dotsc ,n-1\}$. 
\end{theorem}

\begin{proof}
Suppose $T$ has no periodic point of prime period $k\in \{1,\dotsc ,n-1\}$.
Then $T$ has no fixed points. Let $x_{0}\in X$, and $x_{i+1}\in Tx_{i}$ for
all $i\geq 0$. Since $x_{i}\notin Tx_{i}$, then $x_{i}\neq x_{i+1}\in Tx_{i}$%
.

For $n=2$ or $n=3$:

Since $T$ has no periodic points of prime period $2$ and no fixed points, we
conclude that the points $x_{i},x_{i+1},x_{i+2}$ are pairwise distinct.

For $n\geq 4$:

The absence of periodic points of prime periods $2$ and $3$ implies that $%
x_{i},x_{i+1},x_{i+2},x_{i+3}$ are also pairwise distinct.

Repeating this reasoning $n - 4$ more times, we deduce that the points 
\begin{equation*}
x_i, x_{i+1}, \ldots, x_{i+n-1}
\end{equation*}
are pairwise distinct for every $i \geq 0$.

Now, letting $p=S(x_{0},x_{1},\dots ,x_{n-1})$, we have 
\begin{eqnarray*}
d(x_{n},x_{n+1}) &\leq &S(x_{n},x_{n+1},\dots ,x_{2n-1})=\sum_{n\leq i<j\leq
2n-1}d(x_{i},x_{j}) \\
&\leq &\sum_{n\leq i<j\leq 2n-1}\delta (Tx_{i-1},Tx_{j-1}) \\
&=&S(Tx_{n-1},Tx_{n},\dots ,Tx_{2n-2}) \\
&\leq &\alpha S(x_{n-1},x_{n},\dots ,x_{2n-2}) \\
&&\vdots \\
&\leq &\alpha ^{n}S(x_{0},x_{1},\dots ,x_{n-1}) \\
&=&\alpha ^{n}p.
\end{eqnarray*}
\end{proof}

Using the triangle inequality: for any positive integers $m,n$ with $m > n$,
we have 
\begin{eqnarray*}
d(x_{m},x_{n}) &\leq& \sum\limits_{k=n}^{m} d(x_{k},x_{k+1}) \\
&\leq& \sum\limits_{k=n}^{m} \alpha^{k} p \\
&=& \alpha^{n} \sum\limits_{k=n}^{\infty} \alpha^{k-n} p \\
&=& \alpha^{n} \frac{1}{1-\alpha} p
\end{eqnarray*}

which implies that $\left\{ x_{n} \right\}_{n \geq 0}$ is a Cauchy sequence
and hence converges in the complete metric space $X$. Let us denote $\lim
x_{n} = x_{\ast}$. Then we have, by the triangle inequality, 
\begin{eqnarray*}
d(x_{n}, T x_{\ast}) &\leq& \delta(T x_{n-1}, T x_{\ast}) \\
&\leq& S(T x_{\ast}, T x_{n-1}, T x_{n-2}, \dots, T x_{2n-1}) \\
&\leq& \alpha S(x_{\ast}, x_{n-1}, x_{n-2}, \dots, x_{2n-1}) \\
&=& \alpha \sum_{n-1 \leq i < j \leq 2n-1}^{k} d(x_i, x_j) + \alpha
\sum_{n-1 \leq i \leq 2n-1}^{k} d(x_i, x_{\ast})
\end{eqnarray*}

which implies $x_{\ast} \in T x_{\ast}$ and contradicts our assumption.

\begin{definition}
A mapping $T: X \rightarrow X$ multivalued contracting perimeters of
triangles if there exists $\alpha \in [0,1)$ such that for all distinct $x,
y, z \in X$, 
\begin{equation*}
\delta(Tx, Ty) + \delta(Ty, Tz) + \delta(Tx, Tz) \leq \alpha (d(x, y) + d(y,
z) + d(x, z)).
\end{equation*}
\end{definition}

\begin{corollary}
\cite{petrov} Let $(X, d)$ be a complete metric space with $|X| \geq n$, and 
$T: X \rightarrow X$ a mapping contracting the total pairwise distances
between $n$ points. Then $T$ has a periodic point of prime period $k \in
\{1, \dotsc, n-1\}$.
\end{corollary}

\begin{corollary}
\cite{petrov2} If $T$ contracts perimeters of triangles, then $T$ has a
fixed point if and only if there are no periodic points of period 2.
\end{corollary}

\begin{corollary}
\cite{petrov2} If $T$ multivalued contracting perimeters of triangles, then $%
T$ has a fixed point if and only if there are no periodic points of period 2.
\end{corollary}

\begin{corollary}
\cite{banach} A contraction mapping on a nonempty complete metric space has
a unique fixed point.
\end{corollary}

\section{Generalized Orbital Triangular Contractions}

\begin{definition}
\cite{popescu} Let $(X, d)$ be a metric space. A mapping $T: X \rightarrow X$
is a generalized orbital triangular contraction if there exists $\alpha \in
[0,1)$ such that: 
\begin{equation*}
d(Tx, T^{2}x) + d(T^{2}x, Ty) + d(Ty, Tx) \leq \alpha \left[ d(x, Tx) +
d(Tx, y) + d(y, x) \right],
\end{equation*}
for all $x, y \in X$ with $x \neq y \neq Tx$.
\end{definition}

\begin{definition}
Let $(X, d)$ be a metric space. A mapping $T: X \rightarrow CB(X)$ is a
multivalued generalized orbital triangular contraction if there
exists $\alpha \in [0,1)$ such that: 
\begin{equation*}
\delta(Tx, T^{2}x) + \delta(T^{2}x, Ty) + \delta(Ty, Tx) \leq \alpha \left[
d(x, Tx) + d(Tx, y) + d(x, y) \right],
\end{equation*}
for all $x, y \in X$ with $x \neq y \notin Tx$.
\end{definition}

\begin{theorem}
Let $(X, d)$ be a complete metric space and $T: X \rightarrow X$ a
multivalued generalized orbital triangular contraction with no periodic
points of prime period 2. Then $T$ has a fixed point. 
\end{theorem}

\begin{proof}
Let $x_{0} \in X$ and define the Picard iteration $x_{n+1} \in T x_{n}$.
Assume $x_{n-1} \neq x_{n} \neq x_{n+1} \notin T x_{n}, x_{n} \notin T
x_{n-1}$ for all $n$. Let $x = x_{n-1}, y = x_{n+1}$. Then, we have 
\begin{eqnarray*}
d(x_{n}, x_{n+1}) &\leq& d(x_{n}, x_{n+1}) + d(x_{n+1}, x_{n+2}) + d(x_{n},
x_{n+2}) \\
&\leq& \delta(T x_{n-1}, T^{2} x_{n-1}) + \delta(T^{2} x_{n-1}, T x_{n+1}) +
\delta(T x_{n-1}, T x_{n+1}) \\
&\leq& \alpha \left[ d(x_{n-1}, T x_{n-1}) + d(T x_{n-1}, x_{n+1}) +
d(x_{n-1}, x_{n+1}) \right] \\
&\leq& \alpha \left[ d(x_{n-1}, x_{n}) + d(x_{n}, x_{n+1}) + d(x_{n-1},
x_{n+1}) \right] \\
&\leq& \alpha^{n} \left[ d(x_{0}, x_{1}) + d(x_{1}, x_{2}) + d(x_{0}, x_{2}) %
\right] \\
&=& \alpha^{n} p
\end{eqnarray*}
where $p = d(x_{0}, x_{1}) + d(x_{1}, x_{2}) + d(x_{0}, x_{2})$.

For any positive integers $m,n$ with $m > n$, we have 
\begin{eqnarray*}
d(x_{m}, x_{n}) &\leq& \sum\limits_{k=n}^{m} d(x_{k}, x_{k+1}) \\
&\leq& \sum\limits_{k=n}^{m} \alpha^{k} p \\
&\leq& \alpha^{n} \sum\limits_{k=n}^{\infty} \alpha^{k-n} p \\
&=& \alpha^{n} \frac{1}{1 - \alpha} p
\end{eqnarray*}
which implies that $\left\{ x_{n} \right\}_{n \geq 0}$ is a Cauchy sequence
and hence converges in the complete metric space $X$. Let $\lim x_{n} =
x_{\ast}$. Then we have 
\begin{eqnarray*}
&& d(x_{n}, x_{n+1}) + d(x_{n+1}, T x_{\ast}) + d(x_{n}, T x_{\ast}) \\
&\leq& \delta(T x_{n-1}, T^{2} x_{n-1}) + \delta(T^{2} x_{n-1}, T x_{\ast})
+ \delta(T x_{n-1}, T x_{\ast}) \\
&\leq& \alpha \left[ d(x_{n-1}, T x_{n-1}) + d(T x_{n-1}, x_{\ast}) +
d(x_{n-1}, x_{\ast}) \right] \\
&\leq& \alpha \left[ d(x_{n-1}, x_{n}) + d(x_{n}, x_{\ast}) + d(x_{n-1},
x_{\ast}) \right].
\end{eqnarray*}
Then we have 
\begin{equation*}
d(x_{\ast}, T x_{\ast}) = 0.
\end{equation*}
\end{proof}

\begin{corollary}
\cite{popescu} Let $(X,d)$ be a complete metric space and $T:X\rightarrow X$
a generalized orbital triangular contraction with no periodic points of
prime period 2. Then $T$ has a unique fixed point.
\end{corollary}

\section{Generalized Orbital Triangular Kannan Contractions}

\begin{definition}
\cite{popescu} A mapping $T:X\rightarrow X$ is a generalized orbital
triangular Kannan contraction if there exists $\beta \in \left[ 0,\frac{2}{3}%
\right) $ such that: 
\begin{equation*}
d(Tx,T^{2}x)+d(T^{2}x,Ty)+d(Ty,Tx)\leq \beta \left[
d(x,Tx)+d(y,Ty)+d(Tx,T^{2}x)\right],
\end{equation*}
for all $x,y\in X$ with $x,y,Tx$ pairwise distinct.
\end{definition}

\begin{definition}
A mapping $T:X\rightarrow CB(X)$ is a multivalued generalized orbital
triangular Kannan contraction if there exists $\beta \in \left[ 0,\frac{2}{3}%
\right) $ such that: 
\begin{equation*}
\delta (Tx,T^{2}x)+\delta (T^{2}x,Ty)+\delta (Ty,Tx)\leq \beta \left[
d(x,Tx)+d(y,Ty)+\delta (Tx,T^{2}x)\right],
\end{equation*}
for all $x,y\in X$ with $x\neq y$ and $x,y\notin Tx$ pairwise distinct.
\end{definition}

\begin{remark}
A mapping $T:X\rightarrow CB(X)$ is a multivalued generalized orbital
triangular Kannan contraction satisfies 
\begin{equation*}
(1-\beta )\delta (Tx,T^{2}x)+\delta (T^{2}x,Ty)+\delta (Ty,Tx)\leq \beta 
\left[ d(x,Tx)+d(y,Ty)\right].
\end{equation*}
\end{remark}

\begin{theorem}
Let $(X, d)$ be a complete metric space and $T: X \rightarrow X$ a multivalued generalized orbital triangular
Kannan contraction with no periodic points of prime period 2. Then $T$ has a fixed point. 
\end{theorem}

\begin{proof}
Let $x_{0}\in X$ and define the Picard iteration $x_{n+1}\in Tx_{n}$. Assume 
$x_{n}\neq x_{n+1}\neq x_{n+2}$ for all $n$. Let $x=x_{n},y=x_{n+2}$. Then
we have 
\begin{eqnarray*}
&&(1-\beta )d(x_{n+1},x_{n+2})+d(x_{n+2},x_{n+3})+d(x_{n+3},x_{n+1}) \\
&\leq &(1-\beta )\delta (Tx_{n},T^{2}x_{n})+\delta
(T^{2}x_{n},Tx_{n+2})+\delta (Tx_{n+2},Tx_{n}) \\
&\leq &\beta \left[ d(x_{n},Tx_{n})+d(x_{n+2},Tx_{n+2})\right]  \\
&\leq &\beta \left[ d(x_{n},x_{n+1})+d(x_{n+2},x_{n+3})\right] ,
\end{eqnarray*}%
which implies that 
\begin{equation*}
(1-\beta )\left( d(x_{n+1},x_{n+2})+d(x_{n+2},x_{n+3})\right)
+d(x_{n+3},x_{n+1})\leq \beta d(x_{n},x_{n+1}).
\end{equation*}%
And we have 
\begin{eqnarray*}
(1-\beta )\left( d(x_{n+1},x_{n+2})+d(x_{n+2},x_{n+3})\right) +\left\vert
d(x_{n+2},x_{n+1})-d(x_{n+2},x_{n+3})\right\vert  &\leq & \\
(1-\beta )\left( d(x_{n+1},x_{n+2})+d(x_{n+2},x_{n+3})\right)
+d(x_{n+3},x_{n+1}) &\leq &\beta d(x_{n},x_{n+1}).
\end{eqnarray*}%
If $d(x_{n+2},x_{n+1})-d(x_{n+2},x_{n+3})\geq 0$ then we have 
\begin{equation*}
(2-\beta )d(x_{n+1},x_{n+2})-\beta d(x_{n+2},x_{n+3})\leq \beta
d(x_{n},x_{n+1}),
\end{equation*}%
which implies 
\begin{equation*}
d(x_{n+1},x_{n+2})\leq \frac{\beta }{2-\beta }d(x_{n},x_{n+1}).
\end{equation*}

If $d(x_{n+2},x_{n+3})-d(x_{n+2},x_{n+1})>0$ then we have 
\begin{equation*}
(2-\beta )d(x_{n+1},x_{n+2})\leq (2-\beta )d(x_{n+2},x_{n+3})\leq \beta
(d(x_{n},x_{n+1})+d(x_{n+2},x_{n+1})),
\end{equation*}
which implies 
\begin{equation*}
d(x_{n+1},x_{n+2})\leq \frac{\beta }{2-\beta }d(x_{n},x_{n+1}).
\end{equation*}
Therefore, if we let $\alpha =\frac{\beta }{2-\beta }<1$ we have 
\begin{eqnarray*}
d(x_{n+1},x_{n+2}) &\leq &\alpha d(x_{n},x_{n+1}) \\
&&\dots \\
&\leq &\alpha ^{n+1}d(x_{0},x_{1}).
\end{eqnarray*}
For any positive integers $m,n$ with $m>n$ we have 
\begin{eqnarray*}
d(x_{m},x_{n}) &\leq &\sum\limits_{k=n}^{m}d(x_{k},x_{k+1}) \\
&\leq &\sum\limits_{k=n}^{\infty }\alpha ^{k+1}d(x_{0},x_{1}) \\
&\leq &\alpha ^{n+1}\sum\limits_{k=n}^{\infty }\alpha ^{k-n}d(x_{0},x_{1}) \\
&=&\alpha ^{n+1}\frac{1}{1-\alpha }d(x_{0},x_{1}),
\end{eqnarray*}
which implies that $\left\{ x_{n}\right\} _{n\geq 0}$ is a Cauchy sequence
and it converges in the complete metric space $X$. Let us say $\lim
x_{n}=x_{\ast}$. Then we have 
\begin{eqnarray*}
&&(1-\beta )d(x_{n+1},x_{n+2})+d(x_{n+2},Tx_{\ast})+d(Tx_{\ast},x_{n+1}) \\
&\leq &(1-\beta )\delta (Tx_{n},T^{2}x_{n})+\delta (T^{2}x_{n},Tx_{\ast})
+\delta (Tx_{\ast},Tx_{n}) \\
&\leq &\beta \left[ d(x_{n},Tx_{n})+d(x_{\ast},Tx_{\ast})\right] \\
&\leq &\beta \left[ d(x_{n},x_{n+1})+d(x_{\ast},Tx_{\ast})\right].
\end{eqnarray*}
Then we have 
\begin{equation*}
(2-\beta )d(x_{\ast},Tx_{\ast})\leq 0.
\end{equation*}

Hence $x_{\ast}\in Tx_{\ast}.$

%
\end{proof}

\begin{corollary}
\cite{popescu} In a complete metric space, every generalized orbital
triangular Kannan contraction without periodic points of period 2 has a
fixed point.
\end{corollary}

\section{Generalized Orbital Triangular Chatterjea Contractions}

\begin{definition}
\cite{popescu} A mapping $T:X\rightarrow X$ is a generalized orbital
triangular Chatterjea contraction if there exists $\gamma \in \left( 0,\frac{%
1}{2}\right) $ such that: 
\begin{eqnarray*}
&&d(Tx,T^{2}x)+d(T^{2}x,Ty)+d(Ty,Tx) \\
&\leq &\gamma \left[ d(x,Ty)+d(y,Tx)+d(x,T^{2}x)+d(y,T^{2}x)+d(Tx,Ty)\right]
.
\end{eqnarray*}
for all $x,y\in X$ with $x\neq y\notin Tx$ pairwise distinct.
\end{definition}

\begin{definition}
A mapping $T:X\rightarrow CB(X)$ is a multivalued generalized orbital
triangular Chatterjea contraction if there exists $\gamma \in \left( 0,\frac{%
1}{2}\right) $ such that: 
\begin{eqnarray*}
&&\delta (Tx,T^{2}x)+\delta (T^{2}x,Ty)+\delta (Ty,Tx) \\
&\leq &\gamma \left[ d(x,Ty)+d(y,Tx)+d(x,T^{2}x)+d(y,T^{2}x)+\delta (Tx,Ty)%
\right] .
\end{eqnarray*}
\end{definition}

\begin{remark}
A mapping $T:X\rightarrow CB(X)$ is a multivalued generalized orbital
triangular Chatterjea contraction if it satisfies the following: 
\begin{eqnarray*}
&&\delta (Tx,T^{2}x)+\delta (T^{2}x,Ty)+(1-\gamma )\delta (Ty,Tx) \\
&\leq &\gamma \left[ d(x,Ty)+d(y,Tx)+d(x,T^{2}x)+d(y,T^{2}x)\right] .
\end{eqnarray*}
\end{remark}

\begin{theorem}
Let $(X, d)$ be a complete metric space and $T: X \rightarrow X$ a multivalued generalized orbital triangular
Chatterjea contraction with no periodic points of prime period 2. Then $T$ has a fixed point. 
\end{theorem}

\begin{proof}
Let $x_{0}\in X$ and define the Picard iteration $x_{n+1}\in Tx_{n}$. Assume 
$x_{n}\neq x_{n+1}\neq x_{n+2}$ for all $n$. Let $x=x_{n},y=x_{n+2}$. Then
we have: 
\begin{eqnarray*}
&&d(x_{n+1},x_{n+2})+\delta (x_{n+2},x_{n+3})+(1-\gamma )d(x_{n+3},x_{n+1})
\\
&\leq &\delta (Tx_{n},T^{2}x_{n})+\delta (T^{2}x_{n},Tx_{n+2})+(1-\gamma
)\delta (Tx_{n+2},Tx_{n}) \\
&\leq &\gamma \left[ 
\begin{array}{c}
d(x_{n},Tx_{n+2})+d(x_{n+2},Tx_{n}) \\ 
+d(x_{n},T^{2}x_{n})+d(x_{n+2},T^{2}x_{n})%
\end{array}%
\right] \\
&\leq &\gamma \left[ 
\begin{array}{c}
d(x_{n},x_{n+3})+d(x_{n+2},x_{n+1}) \\ 
+d(x_{n},x_{n+2})%
\end{array}%
\right]
\end{eqnarray*}%
which implies that: 
\begin{eqnarray*}
&&d(x_{n+1},x_{n+2})+\delta (x_{n+2},x_{n+3})+d(x_{n+3},x_{n+1}) \\
&\leq &\delta (Tx_{n},T^{2}x_{n})+\delta (T^{2}x_{n},Tx_{n+2})+(1-\gamma
)\delta (Tx_{n+2},Tx_{n}) \\
&\leq &\gamma \left[ 
\begin{array}{c}
d(x_{n},Tx_{n+2})+d(x_{n+2},Tx_{n}) \\ 
+d(x_{n},T^{2}x_{n})+d(x_{n+2},T^{2}x_{n})%
\end{array}%
\right] \\
&\leq &\gamma \left[ 
\begin{array}{c}
d(x_{n},x_{n+3})+d(x_{n+2},x_{n+1}) \\ 
+d(x_{n},x_{n+2})+d(x_{n+3},x_{n+1})%
\end{array}%
\right] \\
&\leq &\gamma \left[ 
\begin{array}{c}
d(x_{n},x_{n+1})+d(x_{n+2},x_{n+1})+d(x_{n},x_{n+2}) \\ 
+d(x_{n+1},x_{n+2})+\delta (x_{n+2},x_{n+3})+d(x_{n+3},x_{n+1})%
\end{array}%
\right]
\end{eqnarray*}%
Therefore, if we let $\alpha =\frac{\gamma }{1-\gamma }<1$, we have 
\begin{eqnarray*}
d(x_{n+1},x_{n+2})+\delta (x_{n+2},x_{n+3})+d(x_{n+3},x_{n+1}) &\leq &\alpha 
\left[ d(x_{n},x_{n+1})+d(x_{n+2},x_{n+1})+d(x_{n},x_{n+2})\right] \\
&&\cdots \\
&\leq &\alpha ^{n+1}\left[ d(x_{0},x_{1})+d(x_{2},x_{1})+d(x_{0},x_{2})%
\right]
\end{eqnarray*}%
Let $p=\left[ d(x_{0},x_{1})+d(x_{2},x_{1})+d(x_{0},x_{2})\right] $. For any
positive integers $m,n$ with $m>n$ we have 
\begin{eqnarray*}
d(x_{m},x_{n}) &\leq &\sum\limits_{k=n}^{m}d(x_{k},x_{k+1}) \\
&\leq &\sum\limits_{k=n}^{m}\alpha ^{k+1}p \\
&\leq &\alpha ^{n+1}\sum\limits_{k=n}^{\infty }\alpha ^{k-n}p \\
&=&\alpha ^{n+1}\frac{1}{1-\alpha }p
\end{eqnarray*}%
which implies that $\left\{ x_{n}\right\} _{n\geq 0}$ is a Cauchy sequence
and it converges in the complete metric space $X$. Let us say $\lim
x_{n}=x_{\ast }$. Then we have 
\begin{eqnarray*}
&&d(x_{n+1},x_{n+2})+d(x_{n+2},Tx_{\ast })+(1-\gamma )d(Tx_{\ast },x_{n+1})
\\
&\leq &\delta (Tx_{n},T^{2}x_{n})+\delta (T^{2}x_{n},Tx_{\ast })+(1-\gamma
)\delta (Tx_{\ast },Tx_{n}) \\
&\leq &\gamma \left[ d(x_{n},Tx_{\ast })+d(x_{\ast
},Tx_{n})+d(x_{n},T^{2}x_{n})+d(x_{\ast },T^{2}x_{n})\right] \\
&\leq &\gamma \left[ d(x_{n},Tx_{\ast })+d(x_{\ast
},x_{n+1})+d(x_{n},x_{n+2})+d(x_{\ast },x_{n+2})\right]
\end{eqnarray*}%
Then we have 
\begin{equation*}
(2-2\gamma )d(x_{\ast },Tx_{\ast })\leq 0.
\end{equation*}%
Hence $x_{\ast }\in Tx_{\ast }$.
\end{proof}

\begin{corollary}
\cite{popescu} In a complete metric space, every generalized orbital
triangular Kannan contraction without periodic points of period 2 has a
fixed point.
\end{corollary}


\begin{thebibliography}{9}
\bibitem{kannan1968} R. Kannan, "Some results on fixed point theory," 
\textit{Bulletin of the Malaysian Mathematical Society (2nd Series)}, vol.
41, no. 1, pp. 1--16, 1968.

\bibitem{chatterjea1972} P. K. Chatterjea, "On some contraction mappings in
metric spaces," \textit{Bulletin of the Calcutta Mathematical Society}, vol.
64, no. 1, pp. 93--99, 1972.

\bibitem{banach} S. Banach, \emph{Sur les op\'{e}rations dans les ensembles
abstraits et leur application aux \'{e}quations int\'{e}grales}, Fund.
Math., 3:133--181, 1922.

\bibitem{petrov} Petrov, E. (2025). Periodic points of mappings contracting
total pairwise distances. Journal of Fixed Point Theory and Applications,
27(2), 1-11.

\bibitem{petrov2} Petrov, E. (2023). Fixed point theorem for mappings
contracting perimeters of triangles. Journal of Fixed Point Theory and
Applications, 25(3), 74.

\bibitem{popescu} Pacurar, C. M., \& Popescu, O. (2024). Fixed points for
three point generalized orbital triangular contractions. arXiv preprint
arXiv:2404.15682.
\end{thebibliography}
\end{document}